\documentclass[12pt,a4paper]{amsart} 

\usepackage{latexsym}
\usepackage{amsmath}
\usepackage{amsthm}
\usepackage{tikz}
\usepackage{latexsym}
\usepackage{amssymb}
\usepackage[a4paper , lmargin = {2cm} , rmargin = {2cm} , tmargin = {2.5cm} , bmargin = {2.5cm} ]{geometry} 

\usepackage{caption}
\usepackage{hyperref}

\renewcommand{\phi}{\varphi}

\makeatletter
\def\theenumi{\@roman\c@enumi}
\makeatother

\theoremstyle{plain}
\newtheorem*{NewPropositionA}{Proposition A}	
\newtheorem*{NewLemmaB}{Lemma B}	
\newtheorem*{NewPropositionB}{Proposition C}
\newtheorem*{NewTheoremC}{Theorem D}
\newtheorem*{NewConjecture}{Conjecture}
\newtheorem*{NewPropositionE}{Proposition E}

\newtheorem{theorem}[subsection]{Theorem}

\newtheorem{lemma}[subsection]{Lemma}

\newtheorem{examples}[subsection]{Examples}
\newtheorem{proposition}[subsection]{Proposition}
\theoremstyle{definition}
\newtheorem{definition}[subsection]{Definition}

\title[]{A note on semicompleteness of graph products of abelian groups}
\author{Philip M\"oller and Olga Varghese}
\date{\today}

\address{Philip M\"oller\\
	Department of Mathematics\\
	University of M\"unster\\ 
	Einsteinstra\ss e 62\\
	48149 M\"unster (Germany)}
\email{philip.moeller@uni-muenster.de}

\address{Olga Varghese\\
	Department of Mathematics\\
	Otto-von-Guericke University of Magdeburg\\ 
	Universit\"atsplatz 2\\
	39106 Magdeburg (Germany)}
\email{olga.varghese@ovgu.de}

\begin{document}
\begin{abstract} 
In this short note we prove that a graph product $G_\Gamma$ of finitely generated abelian groups is semicomplete -- that is the kernel of the natural homomorphism ${\rm Aut}(G_\Gamma)\to{\rm Aut}(G_\Gamma^{ab})$ induced by the abelianization of $G_\Gamma$ is equal to the inner automorphisms -- if and only if $\Gamma$ does not have a separating star.

\vspace{1cm}
\hspace{-0.6cm}
\medskip

\medskip
\hspace{-0.5cm}{\bf 2010 Mathematics Subject Classification.} Primary: 20E36; Secondary: 20F65.
\end{abstract}

\thanks{The first author is funded
	by a stipend of the Studienstiftung des deutschen Volkes and  by the Deutsche Forschungsgemeinschaft (DFG, German Research Foundation) under Germany's Excellence Strategy EXC 2044--390685587, Mathematics M\"unster: Dynamics-Geometry-Structure. The second author is supported by DFG grant VA 1397/2-1. This work is part of the PhD project of the first author.}

\maketitle

\section{Introduction}

Let $G$ be a group and ${\rm Aut}(G)$ the automorphism group of $G$. We denote by $\Phi\colon{\rm Aut}(G)\to{\rm Aut}(G^{ab})$ the natural homomorphism  induced from
the abelianization map $G\twoheadrightarrow G^{ab}$. Following Bachmuth \cite{Bachmuth} we call the kernel of this map IA$(G):=\ker(\Phi)$. In the abbreviation IA the I stands for identity and the A for abelianization. The group IA$(G)$ contains the group of inner automorphisms of $G$, but it is in general much larger than ${\rm Inn}(G)$. Hence, the group IA$(G)$ reflects the complexity of the algebraic structure of ${\rm Aut}(G)$. 
Many groups in geometric group theory have rigid automorphism groups, in the sense that IA$(G)={\rm Inn}(G)$.
By definition, a group $G$ is called \emph{semicomplete} if IA$(G)={\rm Inn}(G)$. 

Let us discuss  two extreme cases: the subgroup IA$(G)$ is trivial if and only if $G$ is abelian. Hence, all abelian groups are semicomplete. The second extreme case is where ${\rm Aut}(G)={\rm Inn}(G)$. This case has a connection to the stronger rigidity notion of \emph{completeness}. A group $G$ is called \emph{complete} if $G$ has trivial center and every automorphism of $G$ is inner. The notion of completeness goes back to  H\"older \cite{Hoelder}, where he studied decompositions of a group $G$. More precisely, given a group $G$ and a normal subgroup $N$ of $G$ one can ask the question how the group $N$ is involved in the decompositions of a group $G$ into smaller pieces where one puzzle piece is equal to the group $N$. H\"older proved that if $N$ is  complete, then any short exact sequence 
$1\to N\to G\to L\to 1$
splits and $G$ is isomorphic to the direct product $N\times L$. 

Clearly, if $G$ is complete, then $G$ is also semicomplete. An example of a group that is semicomplete but not complete is the free group of rank two $F_2:= F(\left\{x_1, x_2\right\})$. The automorphism induced by the map that maps $x_1$ to $x_2$ and $x_2$ to $x_1$ is clearly not inner. Nielsen showed in \cite{Nielsen} that the free group $F_2$ is semicomplete. In the case where the rank of the free group is larger than two the group of IA-automorphisms is  much larger than ${\rm Inn}(F_n)$. In particular, the group IA$(F_3)$ is not even finitely presentable \cite{McCool}. Free groups are special cases of free products where the puzzle pieces are infinite cyclic groups. The semicompleteness of free products of arbitrary groups was studied by Andreadakis in \cite{Andreadakis}. He showed that a free product of two groups $A$ and $B$ is semicomplete if and only if $A$ and $B$ are both abelian. Therefore a free product of more than two groups is never semicomplete. 

Given groups $G_1,\ldots, G_n$, the free product construction $G_1*G_2*\ldots *G_n$ is one tool to obtain a new group out of the given groups. Graph products of groups generalize this concept by building new group out of vertex labeled finite graphs where the vertices are labeled by the given groups $G_1,\ldots, G_n$. Given a finite simplicial graph $\Gamma$ and a collection of non-trivial groups $\{ G_u \mid u \in V(\Gamma) \}$ indexed by the vertex-set $V(\Gamma)$ of $\Gamma$, the \emph{graph product} $G_\Gamma$ is defined as the quotient
$$ \left( \underset{u \in V(\Gamma)}{\ast} G_u \right) / \langle \langle [g,h]=1, \ g \in G_u, h \in G_v, \{u,v\} \in E(\Gamma) \rangle \rangle$$
where $E(\Gamma)$ denotes the edge-set of $\Gamma$. These
groups were introduced by Baudisch in \cite{Baudisch} for infinite cyclic vertex groups and later by Green for arbitrary vertex groups \cite{Green}. 

Given a graph $\Gamma$ and a vertex $v\in V(\Gamma)$, we define the \emph{star of $v$} as follows: $st(v):=\left\{v\right\}\cup \left\{w\in V(\Gamma)\mid \left\{v,w\right\}\in E(\Gamma)\right\}$. By definition, a graph $\Gamma$ has a \emph{separating star} if there exists a vertex $v\in V(\Gamma)$ such that the subgraph spanned by the vertex set $V(\Gamma)-st(v)$ is disconnected. For example, if $\Gamma$ has more than two connected components, then $\Gamma$ has a separating star. 

Many group theoretical properties of graph products of groups and their automorphism groups have been translated in combinatorial structure of the
defining graph $\Gamma$, see for example \cite{LohreySenizergues},\cite{Meier}, \cite{Mihalik}, \cite{PaoliniShelah}, \cite{Varghese00}, \cite{Varghese0}, \cite{Varghese1}, \cite{Varghese2}. An interesting question is how combinatorial properties of $\Gamma$ influence the completeness of $G_\Gamma$. 

\begin{NewPropositionA}
	Let $G_\Gamma$ be a graph product of directly indecomposable groups such that the center of the vertex groups is non-trivial. The following statements are equivalent:
	\begin{enumerate}
		\item The graph product $G_\Gamma$ is complete.
		\item All vertex groups are isomorphic to $\mathbb{Z}/2\mathbb{Z}$, $\Gamma$ has no separating star, $\Gamma$ is asymmetric and for any pair of vertices $v, w\in V(\Gamma), v\neq w$,
		$st(v)$ is never a subset of $st(w)$. Further $\Gamma$ is connected and contains at least $7$ vertices.
	\end{enumerate}
\end{NewPropositionA}
We note that the assumption that vertex groups are directly indecomposable is easy to satisfy. If a vertex group splits as a direct product, we can just replace the corresponding vertex with a clique corresponding to the direct decomposition.

An example of a graph satisfying the conditions in Proposition A is the so called Frucht graph $\mathcal{F}$ introduced in \cite{Frucht}, see Figure 1.
If we assign a cyclic group of order two to each vertex of $\mathcal{F}$, then $G_\mathcal{F}$ is complete.

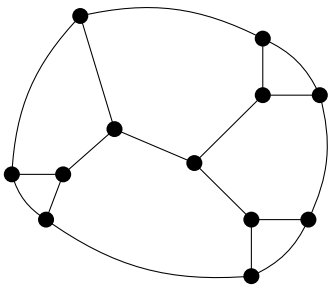
\begin{figure}[h]
	\begin{center}
	\captionsetup{justification=centering}
\begin{tikzpicture}[scale=1.5]

	\coordinate (A) at (-1,1.3);
	\coordinate (B) at (0.6,1.1);
	\coordinate (C) at (1.1,0.6);
	\coordinate (D) at (-1.15,-0.1);
	\coordinate (E) at (-1.3,-0.5);
	\coordinate (F) at (1,-0.5);
	\coordinate (G) at (0.5,-1);
	\coordinate (H) at (-1.6,-0.1);
	\coordinate (I) at (-0.7,0.3);
	\coordinate (J) at (0.6,0.6);
	\coordinate (K) at (0.5,-0.5);
	\coordinate (L) at (0,0);

	\draw[bend right=20] (A) to (H);
	\draw[bend right=20] (H) to (E);
	\draw[bend right=20] (E) to (G);
	\draw[bend right=20] (G) to (F);
	\draw[bend right=20] (F) to (C);
	\draw[bend right=20] (C) to (B);
	\draw[bend right=20] (B) to (A);
	\draw (A) to (I);
	\draw (I) to (D);
	\draw (D) to (H);
	\draw (D) to (E);
	\draw (I) to (L);
	\draw (L) to (K);
	\draw (K) to (F);
	\draw (K) to (G);
	\draw (J) to (B);
	\draw (J) to (C);
	\draw (J) to (L);
	
	\fill[color=black] (A) circle (2pt);
	\fill[color=black] (B) circle (2pt);
	\fill[color=black] (C) circle (2pt);
	\fill[color=black] (D) circle (2pt);
	\fill[color=black] (E) circle (2pt);
	\fill[color=black] (F) circle (2pt);
	\fill[color=black] (G) circle (2pt);
	\fill[color=black] (H) circle (2pt);
	\fill[color=black] (I) circle (2pt);
	\fill[color=black] (J) circle (2pt);
	\fill[color=black] (K) circle (2pt);
	\fill[color=black] (L) circle (2pt);
	
\end{tikzpicture}
\caption{Frucht graph $\mathcal{F}$.}
\end{center}
\end{figure}

The next goal of this note is to give a characterization of those graph products of groups in terms of combinatorial structure of graphs that are semicomplete. Our first result concerning semicompleteness is the following lemma.
\begin{NewLemmaB}(see Lemmata \ref{hasSeparatingStar}, \ref{abelian})
	Let $G_\Gamma$ be a graph product of groups.
	\begin{enumerate}
	 	\item If $\Gamma$ has a separating star, then $G_\Gamma$ is not semicomplete.	
		\item Assume that $V(\Gamma)$ is not a star of a vertex. If $G_\Gamma$ is semicomplete, then all vertex groups are abelian.
		\end{enumerate}
\end{NewLemmaB}
We note that there exist semicomplete graph products where the graph is a star of a vertex and not all vertex groups are abelian. For example, let us discuss the graph product defined via the graph $\Gamma$ in Figure 2.

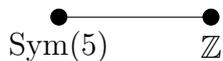
\begin{figure}[h]
	\begin{center}
	\captionsetup{justification=centering}
		\begin{tikzpicture}
			\draw[fill=black]  (0,0) circle (3pt);
			\draw[fill=black]  (2,0) circle (3pt);		
			\node at (0,-0.4) {${\rm Sym}(5)$}; 
			\node at (2,-0.4) {$\mathbb{Z}$};		
			\draw (0,0)--(2,0);		
		\end{tikzpicture}
	\caption{Vertex labeled graph $\Gamma$.}
	\end{center}
\end{figure}

The graph product $G_\Gamma$ is a direct product ${\rm Sym}(5)\times\mathbb{Z}$. It is known that ${\rm Aut}({\rm Sym}(5))={\rm Inn}({\rm Sym}(5))\cong{\rm Sym}(5)$. Hence the vertex groups of $G_\Gamma$ are semicomplete. Further the subgroups ${\rm Sym}(5)\times \left\{0\right\}$ and $\left\{1\right\}\times\mathbb{Z}$ are characteristic, therefore  $G_\Gamma$ is also semicomplete. Hence the graph product $G_\Gamma$ is a star of a vertex and is semicomplete, but not all vertex groups are abelian. For more information on semicompleteness of direct product of (finite) groups, see \cite{Panagopoulos}. 

Let us consider one more example of a graph product where the vertex groups are semicomplete.

\begin{figure}[h]
	\begin{center}
	\captionsetup{justification=centering}
		\begin{tikzpicture}
			\draw[fill=black]  (0,0) circle (3pt);
			\draw[fill=black]  (2,0) circle (3pt);
			\draw[fill=black]  (4,0) circle (3pt);		
			\node at (0,-0.4) {${\rm Sym}(5)$}; 
			\node at (2,-0.4) {$\mathbb{Z}/2\mathbb{Z}$};
			\node at (4,-0.4) {$\mathbb{Z}/2\mathbb{Z}$};		
			\draw (0,0)--(2,0);	
			\draw (2,0)--(4,0);		
		\end{tikzpicture}
	\caption{A graph $\Gamma$ with semicomplete vertex groups.}
	\end{center}
\end{figure}
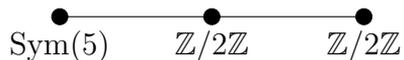
The graph product $G_\Gamma$ defined via the graph in Figure 3 is an amalgamated free product $({\rm Sym}(5)\times\mathbb{Z}/2\mathbb{Z})*_{\mathbb{Z}/2\mathbb{Z}} (\mathbb{Z}/2\mathbb{Z}\times\mathbb{Z}/2\mathbb{Z})$. We know by Andreadakis theorem that the subgroup ${\rm Sym}(5)*\mathbb{Z}/2\mathbb{Z}$ is not semicomplete. It follows directly from the next proposition that $G_\Gamma$ is also not semicomplete. 
\begin{NewPropositionB}\label{star}
	Let $G_\Gamma$ be a graph product of groups. Let $\Delta$ denote the subgraph of $\Gamma$ spanned by all vertices $v\in V(\Gamma)$ such that $st(v)=V(\Gamma)$. 
	
	If the vertex groups $G_v$, $v\in V(\Delta)$ are abelian, then $G_\Gamma$ is semicomplete if and only if the graph product $G_{\langle V(\Gamma)-V(\Delta)\rangle}$ is semicomplete where $\langle V(\Gamma)-V(\Delta)\rangle$ is the subgraph spanned by the vertex set  $V(\Gamma)-V(\Delta)$.
\end{NewPropositionB}

Nevertheless, Lemma B shows that the restriction to abelian vertex groups is in most cases necessary. In this note we mostly assume that the vertex groups are finitely generated abelian groups.
Let $G_\Gamma$ be a graph product of the vertex groups $G_1, \ldots , G_n$ that are finitely generated abelian groups. The abelianization of $G_\Gamma$ is equal to the direct product $G_1\times...\times G_n$ and there exist a finite abelian group $T$ and a natural number $k$ such that $G_1\times...\times G_n\cong\mathbb{Z}^k\times T $. Thus we have the map 
$$\Phi\colon{\rm Aut}(G_\Gamma)\to{\rm Aut}(\mathbb{Z}^k\times T).$$

It was proven in \cite[Lem. 1.16]{SaleSusse} that the automorphism group of $\mathbb{Z}^k\times T$  has the structure of a semidirect product, more precisely ${\rm Aut}(\mathbb{Z}^k\times T)\cong T^k\rtimes ({\rm GL}_k(\mathbb{Z})\times{\rm Aut}(T))$. 
Hence for a semicomplete graph product $G_\Gamma$ of finitely generated abelian groups, $\Phi(f)$ contains all the algebraic information we want about $f\in{\rm Aut}(G_\Gamma)$ up to conjugation. 

The next theorem gives us one more entry in the dictionary between algebraic-geometric properties of $G_\Gamma$ and the combinatorial structure of $\Gamma$. 
\begin{NewTheoremC}\label{fgabelian}	
	Let $G_\Gamma$ be a graph product of finitely generated abelian groups.  The following statements are equivalent:
	\begin{enumerate}
		\item The group $G_\Gamma$ is semicomplete.
		\item The graph $\Gamma$ does not have a separating star.
	\end{enumerate}
\end{NewTheoremC}

The crucial puzzle piece of the proof is the description of the generating set of IA$(G_\Gamma)$ given in \cite{SaleSusse}. In the case where the vertex groups are arbitrary abelian groups we conjecture that the same characterization of semicompleteness holds.

\begin{NewConjecture}
	A graph product $G_\Gamma$ of abelian groups  is semicomplete if and only if $\Gamma$ has no separating star.
\end{NewConjecture}

Using the discription of the automorphism group of $G_\Gamma$ where $\Gamma$ has a special combinatorial structure in \cite[Cor. 8.2]{Genevois} we obtain 
\begin{NewPropositionE}\label{atomic}
Let $G_\Gamma$ be a graph product of groups. Assume that for all vertices $v, w\in V(\Gamma)$, $v\neq w$ we have $lk(v)\nsubseteq st(w)$. The following statements are equivalent:
\begin{enumerate}
\item The group $G_\Gamma$ is semicomplete.
\item All vertex groups are abelian and $\Gamma$ has no separating star.
\end{enumerate}
\end{NewPropositionE}

\subsection*{Acknowledgment} 
We want to thank Anthony Genevois for useful comments on the previous version of this paper and the referee for many helpful remarks. We are also grateful to Dominic Enders for having communicated to
us a proof idea of Proposition \ref{noSIL}.
%%%%%%%%%%%%%%%%%%%%%%%%%%%%%%%%%%%%%%%%%%%%%%%%%%%%%%%%%%%%%%%%%%%%%%%%%%%%%%%%%%%%%%%%%%%%%%%%%%%%%%%%%%%%%%%%%%%%%%%%%%%%%%%%%%%%%%%%%%%%%%%%%%%%%%%%%%
\section{Semicomplete graph products of groups}

\subsection{Characteristic subgroups}

We begin by recalling the concept of characteristic subgroups.
\begin{definition}
	Let $G$ be a group and $N\subseteq G$ be a subgroup. The group $N$ is called \emph{characteristic} in $G$ if for every  $f\in{\rm Aut}(G)$ the equality $f(N)=N$ holds.
\end{definition}

Given a group $G$, by definition the \emph{commutator subgroup} of $G$, denoted by $G'$, is the group generated by the elements $ghg^{-1}h^{-1}$ for $g, h\in G$. Note that $G'$ is normal in $G$ and the quotient $G/G'$ is abelian and is called the \emph{abelianization} of $G$ that we denote by $G^{ab}:=G/G'$.
For example, the abelianization of the free group $F_n$ is the free abelian group $\mathbb{Z}^n$.

\begin{examples}
	Let $G$ be a group.
	\begin{enumerate}
		\item The commutator subgroup $G'$ is characteristic in $G$.
		\item The center of $G$, denoted by $Z(G):=\left\{h\in G\mid hg=gh \text{ for all }g\in G\right\}$, is characteristic in $G$.		
	\end{enumerate}
\end{examples}

Let $G$ be a group and $N$ be a characteristic subgroup. The map $\Psi\colon{\rm Aut}(G)\to{\rm Aut}(G/N)$ defined as follows: $\Psi(f)(xN)=f(x)N$ for $f\in{\rm Aut}(G)$ and $x\in G$ is a well-defined group homomorphism. 

In particular, for a group $G$ we have a group homomorphism 
$\Phi\colon{\rm Aut}(G)\to{\rm Aut}(G^{ab}).$ 
Clearly, $\Phi$ is in general not surjective. Nevertheless, we have always a short exact sequence 
$$\left\{id\right\}\to \ker(\Phi)\to{\rm Aut}(G)\to {\rm im}(\Phi)\to\left\{id\right\}.$$
The subgroup of inner automorphisms of $G$ is always contained in $\ker(\Phi)$. Following \cite{Bachmuth} we denote the kernel of $\Phi$ by IA$(G)$, so ${\rm Inn}(G)\subseteq{\rm IA}(G)$.

For a free group $F_2$ Nielsen showed in \cite{Nielsen} that ${\rm IA}(F_2)={\rm Inn}(F_2)$ and that the map $\Phi$ is also surjective, hence we have the following short exact sequence
$$\left\{1\right\}\to F_2\to{\rm Aut}(F_2)\to{\rm GL}_2(\mathbb{Z})\to\left\{1\right\}.$$

Here we are interested in conditions on the group $G$ such that the automorphism group of $G$ is rigid in the sense that  ${\rm IA}(G)={\rm Inn}(G)$.	

%%%%%%%%%%%%%%%%%%%%%%%%%%%%%%%%%%%%%%%%%%%%%%%%%%%%%%%%%%%%%%%%%%%%%%%%%%%%%%%%%%%%%%%%%%%%%%%%%%%%%%%%%%%%%%%%%%%%%%%%%%%%%%%%%%%%%%%%%%%%%%%%%%%%%%%%%%

\subsection{Semicomplete groups}

\begin{definition}
	Let $G$ be a group and $\Phi\colon{\rm Aut}(G)\to{\rm Aut}(G^{ab})$ be the natural group homomorphism induced by the abelianization of $G$.
	A group $G$ is called \emph{semicomplete} if $\ker(\Phi)={\rm Inn}(G)$.
\end{definition}

Here we investigate semicompleteness of infinite groups that are defined via vertex labeled graphs: \emph{graph products of groups}. An example of such a group is a free product of two groups. The following result was proven by Andreadakis using a characterization of subgroups in free products given by Kurosh subgroup theorem \cite{Kurosh}. 

\begin{theorem}(\cite[Thm. 2, Thm. 3]{Andreadakis})
	Let $G=A*B$ be a free product of two non-trivial groups $A$ and $B$. The group $G$ is semicomplete if and only if $A$ and $B$ are both abelian.
\end{theorem}
Before giving a definition of a graph product of groups we need to recall some facts about simplicial graphs.

%%%%%%%%%%%%%%%%%%%%%%%%%%%%%%%%%%%%%%%%%%%%%%%%%%%%%%%%%%%%%%%%%%%%%%%%%%%%%%%%%%%%%%%%%%%%%%%%%%%%%%%%%%%%%%%%%%%%%%%%%%%%%%%%%%%%%%%%%%%%%%%%%%%%%%%%%%
\subsection{Simplicial graphs}
We first recall a few definitions and some important features of simplicial graphs, following \cite{Diestel}. In our setting, graphs are always without loops and multiple edges and they are always finite. 

Let $V$ denote a non-empty set. A \textit{simplicial graph} $\Gamma$ is a pair $\Gamma=(V,E)$ where $E$ is a set of $2$-element subsets of $V$. The elements of $V$ are called the \textit{vertices} of $\Gamma$ and the elements of $E$ are the \textit{edges} of $\Gamma$. We say that two vertices $v, w$ are \emph{adjacent} if $\left\{v,w\right\}$ is an edge. A simplicial graph $\Gamma$ is called \emph{finite} if the cardinality of $V$ is finite. Given a graph $\Gamma$ with the vertex set $V(\Gamma)$ and the edge set $E(\Gamma)$, for a vertex $v\in V(\Gamma)$, we define the \emph{link of $v$} as follows $lk(v):=\left\{w\in V(\Gamma)\mid \left\{v,w\right\}\in E(\Gamma)\right\}$ and the \emph{star of $v$} is defined as $st(v):=\left\{v\right\}\cup lk(v)$. 
A graph $\Gamma$ is called \emph{clique} if $E(\Gamma)$ contains an edge for every pair of vertices in $V(\Gamma)$.

If $V'\subseteq V$ and $E'\subseteq E$ and $E'$ is a set of 2-element subsets of $V'$, then $\Gamma'=(V', E')$ is called a \emph{subgraph} of $\Gamma$. If $\Gamma'$ is a subgraph of $\Gamma$ and $E'$ contains all the edges $\left\{v, w\right\}\in E$ with $v, w\in V'$, then $\Gamma'$ is called an \emph{induced subgraph} of $\Gamma$. Sometimes we also say that $\Gamma'$ is spanned by the vertex set $V'$. 
A {\it path} is a graph $P_n=(V, E)$ of the form  $V=\left\{v_0, \ldots, v_n\right\}\text{ and }E=\left\{\left\{v_0, v_1\right\}, \left\{v_1, v_2\right\}, \ldots, \left\{v_{n-1}, v_n\right\}\right\}$ where the $v_i$, $0\leq i\leq n$, are pairwise distinct. We say that $P_n$ is a path between $v_0$ and $v_n$. A graph $\Gamma$ is called \emph{connected} if there is always a path between $v$ and $w$ for $v, w\in V(\Gamma)$. An induced subgraph $\Delta$ of $\Gamma$ is called a \emph{connected component} if $\Delta$ is connected and is maximal with respect to inclusion. 
We say that $\Gamma$ has a \emph{separating star (for a vertex $v\in V(\Gamma)$)} if the subgraph spanned by the vertex set $V(\Gamma)-st(v)$ is disconnected. Given a graph $\Gamma$, a \emph{graph automorphism} of $\Gamma$ is a map $f\colon V(\Gamma)\to V(\Gamma)$ such that $\left\{f(v), f(w)\right\}$ is an edge if and only if $\left\{v,w\right\}$ is an edge. A graph is called \emph{asymmetric} if any graph automorphism of $\Gamma$ is trivial.

It was proven in \cite{ErdosRenyi} that an asymmetric graph with more than one vertex has at least $6$ vertices and an asymmetric graph with $6$ vertices is isomorphic to one of the eight graphs in Figure 4. 

Note that an asymmetric graph with 6 vertices always has a separating star 
(consider the stars of empty vertices in Figure 4). 

\newpage
\begin{figure}[h]
	\begin{center}
		\begin{tikzpicture}[scale=1.1]
			\draw (0,0)--(4,0);
			\draw (1,0)--(1.5,1);
			\draw (1.5,1)--(2,0);
			\draw[fill=black]  (0,0) circle (2pt);
			\draw[fill=black]  (1,0) circle (2pt);		
			\draw[fill=white]  (2,0) circle (2pt);
			\draw[fill=black]  (3,0) circle (2pt);		
			\draw[fill=black]  (4,0) circle (2pt);
			\draw[fill=black]  (1.5,1) circle (2pt);

			\draw (6,-1)--(6,1);
			\draw (7,-1)--(7,1);
			\draw (6,1)--(7,1);
			\draw (6,0)--(7,0);
			\draw (6,-1)--(7,0);
			\draw[fill=white] (6,0) circle (2pt);
			\draw[fill=black] (7,0) circle (2pt);
			\draw[fill=black] (6,1) circle (2pt);
			\draw[fill=black] (7,1) circle (2pt);
			\draw[fill=black] (6,-1) circle (2pt);
			\draw[fill=black] (7,-1) circle (2pt);

			\draw (9,-1)--(9,1);
			\draw (9,0)--(10,0);
			\draw (10,0)--(10,1);
			\draw (9,1)--(10,1);
			\draw (9,1)--(9.5,2);
			\draw (10,1)--(9.5,2);
			\draw[fill=black] (9,0) circle (2pt);
			\draw[fill=black] (10,0) circle (2pt);
			\draw[fill=white] (9,1) circle (2pt);
			\draw[fill=black] (10,1) circle (2pt);
			\draw[fill=black] (9,-1) circle (2pt);
			\draw[fill=black] (9.5,2) circle (2pt);
			
			\draw (12,-1)--(12,1);
			\draw (13,-1)--(13,1);
			\draw (12,0)--(13,0);
			\draw (12,1)--(13,1);
			\draw (12,1)--(13,0);
			\draw[fill=black] (12,0) circle (2pt);
			\draw[fill=black] (13,0) circle (2pt);
			\draw[fill=white] (12,1) circle (2pt);
			\draw[fill=black] (13,1) circle (2pt);
			\draw[fill=black] (12,-1) circle (2pt);
			\draw[fill=black] (13,-1) circle (2pt);
			
			\draw (0,-3)--(0,-5);
			\draw (1,-3)--(1,-5);
			\draw (0,-3)--(1,-3);
			\draw (0,-4)--(1,-4);
			\draw (0,-5)--(1,-4);
			\draw (0,-4)--(1,-3);
			\draw[fill=black] (0,-3) circle (2pt);
			\draw[fill=black] (0,-4) circle (2pt);
			\draw[fill=black] (0,-5) circle (2pt);
			\draw[fill=white] (1,-3) circle (2pt);
			\draw[fill=black] (1,-4) circle (2pt);
			\draw[fill=black] (1,-5) circle (2pt);

			\draw (3,-5)--(3,-3);
			\draw (4,-5)--(4,-3);
			\draw (3,-3)--(4,-3);
			\draw (3,-4)--(4,-4);
			\draw (3,-5)--(4,-5);
			\draw (3,-4)--(4,-3);
			\draw[fill=black] (3,-3) circle (2pt);
			\draw[fill=black] (3,-4) circle (2pt);
			\draw[fill=black] (3,-5) circle (2pt);
			\draw[fill=black] (4,-3) circle (2pt);
			\draw[fill=white] (4,-4) circle (2pt);
			\draw[fill=black] (4,-5) circle (2pt);

			\draw (6,-5)--(6,-3);
			\draw (6,-3)--(7,-3);
			\draw (6,-4)--(7,-4);
			\draw (6,-5)--(7,-5);
			\draw (6,-4)--(7,-3);
			\draw (6,-5)--(7,-4);
			\draw (7,-3)--(7,-4);
			\draw[fill=black] (6,-3) circle (2pt);
			\draw[fill=black] (6,-4) circle (2pt);
			\draw[fill=black] (6,-5) circle (2pt);
			\draw[fill=black] (7,-3) circle (2pt);
			\draw[fill=white] (7,-4) circle (2pt);
			\draw[fill=black] (7,-5) circle (2pt);

			\draw (9,-3)--(11,-3);
			\draw (9,-5)--(11,-5);
			\draw (9,-3)--(9,-5);
			\draw (11,-3)--(11,-5);
			\draw (9,-5)--(11,-3);
			\draw (9,-5)--(9.5, -3.5);
			\draw (9.5, -3.5)--(11,-5);
			\draw[fill=black] (9,-3) circle (2pt);
			\draw[fill=black] (11,-3) circle (2pt);
			\draw[fill=black] (9,-5) circle (2pt);
			\draw[fill=white] (11,-5) circle (2pt);	
			\draw[fill=black] (10,-4) circle (2pt);
			\draw[fill=black] (9.5,-3.5) circle (2pt);
		
		\end{tikzpicture}
	\caption{A characterization of asymmetric graphs with 6 vertices.}
	\end{center}
\end{figure}
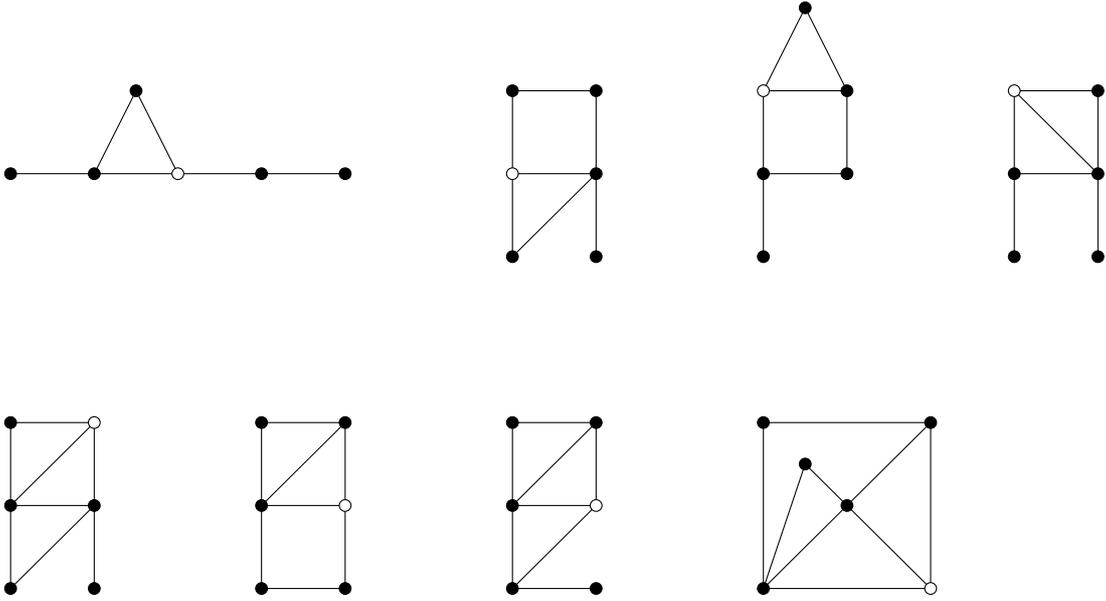

\begin{definition}\label{SIL}
		Given a graph $\Gamma=(V,E)$.
		 A \textit{seperating intersection of links (SIL)} is a triple of vertices $(x,y \mid z)$ in $\Gamma$ that are pairwise non-adjacent and such that the connected component of the subgraph spanned by  $V(\Gamma)-(lk(x)\cap lk(y))$ containing $z$ does not contain either $x$ or $y$. 
\end{definition}

Let us consider the graph in Figure 5. 
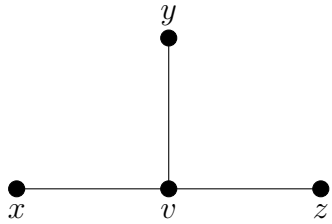
\begin{figure}[h]
	\begin{center}
		\begin{tikzpicture}
			\draw[fill=black]  (0,0) circle (3pt);
			\draw[fill=black]  (2,0) circle (3pt);
			\draw[fill=black]  (4,0) circle (3pt);	
			\draw[fill=black]  (2,2) circle (3pt);	
			\node at (0,-0.3) {$x$}; 
			\node at (2,-0.3) {$v$};
			\node at (4,-0.3) {$z$};
			\node at (2,2.3) {$y$};	
			\draw (0,0)--(2,0);	
			\draw (2,0)--(4,0);	
			\draw (2,0)--(2,2);	
		\end{tikzpicture}
	\caption{A graph $\Gamma_1$ with a SIL and a separating star.}
	\end{center}
\end{figure}

The graph induced by the vertex set $V(\Gamma_1)-lk(x)\cap lk(y)=\left\{x,y,v,z\right\}-\left\{v\right\}=\left\{x, y, z\right\}$ has three connected components each of them is a single vertex. Hence $(x,y\mid z)$ is a SIL. Further the graph in Figure 5 has a separating star for the vertex $x$.

Let us now discuss the graph in Figure 6.
\begin{figure}[h]
	\begin{center}
		\begin{tikzpicture}
			\draw[fill=black]  (0,0) circle (3pt);
			\draw[fill=black]  (2,0) circle (3pt);
			\draw[fill=black]  (4,0) circle (3pt);	
			\draw[fill=black]  (6,0) circle (3pt);	
			\draw[fill=black]  (8,0) circle (3pt);
			\draw[fill=black]  (0,2) circle (3pt);
			\draw[fill=black]  (2,2) circle (3pt);	
			\draw[fill=black]  (4,2) circle (3pt);	
			\draw[fill=black]  (6,2) circle (3pt);	
			\draw[fill=black]  (8,2) circle (3pt);
			\draw (0,0)--(8,0);	
			\draw (0,2)--(8,2);	
			\draw (0,0)--(0,2);	
			\draw (2,0)--(2,2);	
			\draw (4,0)--(4,2);	
			\draw (6,0)--(6,2);	
			\draw (8,0)--(8,2);
			\node at (4,-0.3) {$v$};
		\end{tikzpicture}
	\caption{A graph $\Gamma_2$ with no SIL and a separating star.}
	\end{center}
\end{figure}
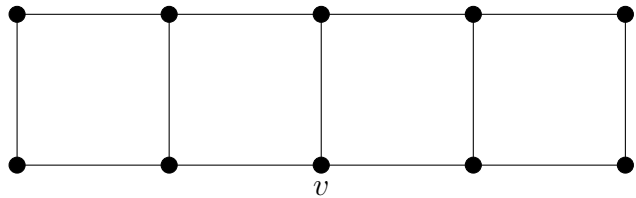

It easy to verify that the graph $\Gamma_2$ in Figure 6 does not have a SIL, but $\Gamma_2$ has a separating star for a vertex $v$, see \cite[p. 255]{Charney}.

For us the following observation about separating stars and SILs is crucial.
\begin{proposition}\label{noSIL}
	If $\Gamma$ has no separating star, then $\Gamma$ has no SIL.
\end{proposition}
\begin{proof}
	  By assumption $\Gamma$ has no separating star, hence $\Gamma$ has at most two connected components. If $\Gamma$ has two connected components and no separating star, then each of the connected components is a clique. Therefore a graph with two connected components and no separating star does not have a SIL. 
	
	  Now let $\Gamma$ be connected.
	We show the contra-positive, that is if $\Gamma$ has a SIL $(x,y\mid z)$, then $\Gamma$ has a separating star.

	 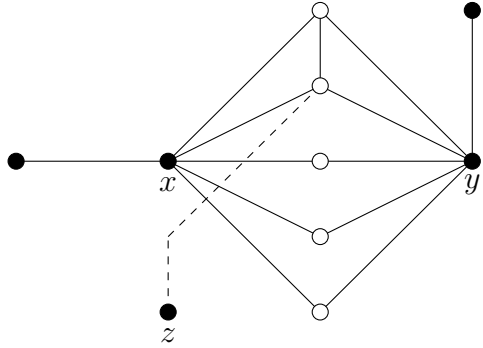
\begin{figure}[h]
	 	\begin{center}
	 	\captionsetup{justification=centering}
	 		\begin{tikzpicture}
	 			\draw (0,0)--(-2,0);
	 			\draw (0,0)--(2,2);
	 			\draw (2,2)-- (2,1);
	 			\draw (0,0)--(2,1);
	 			\draw (0,0)--(2,0);
	 			\draw (4,0)--(4,2);
	 			\draw (4,0)--(2,1);
	 			\draw (4,0)--(2,2);
	 			\draw (4,0)--(2,0);
	 			\draw (4,0)--(2,-1);
	 			\draw (4,0)--(2,-2);
	 				\draw (0,0)--(2,-1);	
	 			\draw (0,0)--(2,-2);
	 			\draw[dashed] (0,-2)--(0,-1);
	 			\draw[dashed] (0,-1)--(2,1); 
	 			
	 			\draw[fill=black]  (0,0) circle (3pt);
	 			\draw[fill=white] (2,0) circle (3pt);
	 			\draw[fill=black]  (4,0) circle (3pt);	
	 			\draw[fill=white]  (2,2) circle (3pt);	
	 			\draw[fill=white]  (2,-2) circle (3pt);
	 			\draw[fill=black]  (0,-2) circle (3pt);
	 			\draw[fill=white]  (2,1) circle (3pt);
	 			\draw[fill=white]  (2,-1) circle (3pt);
	 			
	 			\draw[fill=black]  (-2,0) circle (3pt);
	 			\draw[fill=black]  (4,2) circle (3pt); 	
	 			\node at (0,-0.3) {$x$};
	 			\node at (4,-0.3) {$y$};
	 				\node at (0,-2.3) {$z$};
	 		\end{tikzpicture}
 		\caption{A graph $\Gamma$ with a SIL $(x,y | z)$.}
	 	\end{center}
	 \end{figure}

The vertices $x,y,z$ in the graph $\Gamma$ in Figure 7 form a SIL $(x,y | z)$ and the empty vertices are precisely $lk(x)\cap lk(y)$.
 
 Since $\Gamma$ is connected there exists a path between $z$ and $x$ and a path between $z$ and $y$. But each path between $z$ and $x$ (resp. $z$ and $y$ )  must contain a vertex in $lk(x)\cap lk(y)$, because $z$ is contained in a connected component $C$ of the subgraph induced by the vertex set $V(\Gamma)-lk(x)\cap lk(y)$ and $x\notin V(C)$ and $y\notin V(C)$.
	 
	 Now we claim that $\Gamma$ has a separating star for the vertex $x$. 
	 The subgraph induced by the vertex set $V(\Gamma)-st(x)$ has at least two connected components: one connected component has $y$ as a vertex and the other component has $z$ as a vertex, since each path between $z$ and $y$ in $\Gamma$ has a vertex in $lk(x)\cap lk(y)\subseteq st(x)$. Hence $\Gamma$ has a separating star.
\end{proof}

%%%%%%%%%%%%%%%%%%%%%%%%%%%%%%%%%%%%%%%%%%%%%%%%%%%%%%%%%%%%%%%%%%%%%%%%%%%%%%%%%%%%%%%%%%%%%%%%%%%%%%%%%%%%%%%%%%%%%%%%%%%%%%%%%%%%%%%%%%%%%%%%%%%%%%%%%%
\subsection{Semicomplete graph products of groups}

In this note our focus is on semicompleteness of groups which are constructed from abelian groups.
\begin{definition}
 Given a finite simplicial graph $\Gamma$ and a collection of non-trivial groups $\{ G_v \mid v \in V(\Gamma) \}$ indexed by the vertex-set $V(\Gamma)$ of $\Gamma$, the \emph{graph product} $G_\Gamma$ is defined as the quotient
$$ \left( \underset{v \in V(\Gamma)}{\ast} G_v \right) / \langle \langle [g,h]=1, \ g \in G_v, h \in G_w, \{v,w\} \in E(\Gamma) \rangle \rangle$$
where $E(\Gamma)$ denotes the edge-set of $\Gamma$. 
\end{definition}
 
In order to prove that many graph products of groups are not semicomplete  we need a precise definition of some elements of ${\rm Aut}(G_\Gamma)$. 

\begin{definition}
	Let $G_\Gamma$ be a graph product, $v\in V(\Gamma)$, $x\in G_v$, $x\neq 1$ and $C=(V', E')$ be a connected component of a subgraph spanned by $V(\Gamma)-st(v))$. The \emph{partial conjugation $\pi_{v,x, C}$} is the automorphism of $G_\Gamma$  induced by: 
	\[h\mapsto \left\{
	\begin{array}{ll}
		xhx^{-1} & \text{if }h\in G_w, w\in V' \\
		h & \text{if }h\in G_w, w\notin V'.  \\
	\end{array}
	\right. \]
\end{definition}
We note that an inner automorphism is always a product of partial conjugations.

\begin{lemma}\label{hasSeparatingStar}
	Let $G_\Gamma$ be a graph product of groups. If $\Gamma$ has a separating star, then $G_\Gamma$ is not semicomplete.
\end{lemma}
\begin{proof}
	The partial conjugation $\pi_{v,x, C}$ is not inner exactly when the subgraph spanned by the vertex set  $V(\Gamma)-st(v)$ is not connected. Thus,  if $\Gamma$ has a separating star for a vertex $v$, then the partial conjugation $\pi_{v,x, C}$ lies in IA$(G_\Gamma)$ but this automorphism is not inner. 
\end{proof}

Apart from the partial conjugations, there are more types of automorphisms of $G_\Gamma$, for example \textit{factor automorphisms}. 

\begin{definition}
Let $G_\Gamma$ be a graph product of groups and let $f_v\in{\rm Aut}(G_v)$ be an automorphism of the vertex group $G_v$. A \emph{factor automorphism} $f\in {\rm Aut}(G_\Gamma)$ is  induced by
	\[h\mapsto \left\{
\begin{array}{ll}
	f_v(h) & \text{if } h\in G_v \\
	h &  \text{if }h\in G_w, w\in V(\Gamma), w\neq v.  \\
\end{array}
\right. \]	  
\end{definition}

\begin{lemma}\label{abelian}
	Let $G_\Gamma$ be a graph product of groups. Assume that $V(\Gamma)$ is not a star of a vertex. If $G_\Gamma$ is semicomplete, then all vertex groups $G_v$ are abelian.
\end{lemma}
\begin{proof}
	Assume that $G_v$ is not abelian, then there exists $g\in G_v$ such that $g\notin Z(G_v)$. Then the factor automorphism  that maps $h$ to $ghg^{-1}$ for $h\in G_v$ and $h\mapsto h$ for $h\in G_w$, $w\neq v$ is contained in IA$(G_\Gamma)$ but this automorphism is not inner, since $st(v)\neq V(\Gamma)$.
\end{proof}

\begin{lemma}\label{subgroup}
	If $G_\Gamma$ is semicomplete, then all vertex groups $G_u$ for $u\in V(\Gamma)$ are semicomplete.
\end{lemma}
\begin{proof}
	Assume that there exists a vertex group $G_v$ that is not semicomplete, then there exists an automorphism $f_v\in{\rm Aut}(G_\Gamma)$ such that $f_v\in{\rm IA}(G_v)-{\rm Inn}(G_v)$. Now let $f\in{\rm Aut}(G_\Gamma)$ be the factor automorphism induced by $f_v$. By construction  $f\in{\rm IA}(G_\Gamma)$, but $f$ is clearly not in ${\rm Inn}(G_\Gamma)$.
\end{proof}

We want to remark that the graph product $G_\Gamma$ of finitely generated abelian vertex groups is isomorphic to
the graph product of groups $G_{\Gamma'}$ obtained by replacing each vertex $v$ by a clique with vertices labeled by the cyclic summands of $G_v$ and where the finite cyclic summands have prime power orders. Moreover we make the following observation

\begin{lemma}\label{NoSepStarSIL} Let $G_\Gamma$ denote a graph product of finitely generated abelian groups. Let $G_{\Gamma'}$ denote the isomorphic graph product of cyclic vertex groups where the cyclic vertex groups have prime power order obtained by the above construction. Then $G_\Gamma$ has no seperating star (no SIL) if and only if $G_{\Gamma'}$ has no separating star (no SIL).
\end{lemma}

Let us discuss one small example. The graph on the left in Figure 8 has one vertex $v$ that is a direct product of two groups $\mathbb{Z}/2\mathbb{Z}\times\mathbb{Z}$, so we replace this vertex by an edge $\left\{x,y\right\}$ whose vertices are labeled with the groups $\mathbb{Z}/2\mathbb{Z}$ and $\mathbb{Z}$. We also add edges $\left\{z,x\right\}$ and $\left\{z,y\right\}$ in $\Gamma'$ for all $z\in lk(v)$.

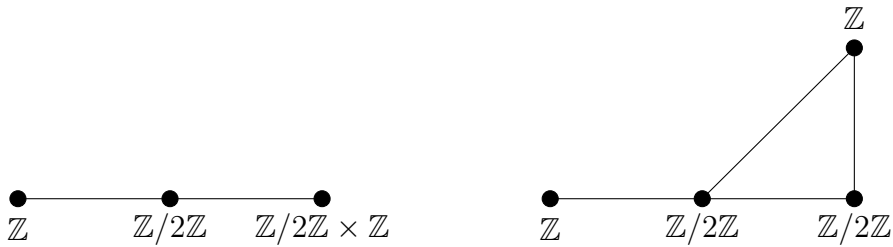
\begin{figure}[h]
	\begin{center}
		\begin{tikzpicture}
			\draw[fill=black]  (0,0) circle (3pt);
			\draw[fill=black]  (2,0) circle (3pt);
			\draw[fill=black]  (4,0) circle (3pt);		
			\node at (0,-0.4) {$\mathbb{Z}$}; 
			\node at (2,-0.4) {$\mathbb{Z}/2\mathbb{Z}$};
			\node at (4,-0.4) {$\mathbb{Z}/2\mathbb{Z}\times\mathbb{Z}$};		
			\draw (0,0)--(2,0);	
			\draw (2,0)--(4,0);

			\draw[fill=black]  (7,0) circle (3pt);
			\draw[fill=black]  (9,0) circle (3pt);
			\draw[fill=black]  (11,0) circle (3pt);	
			\draw[fill=black]  (11,2) circle (3pt);	
			\node at (7,-0.4) {$\mathbb{Z}$};
			 \node at (11,2.4) {$\mathbb{Z}$};
			\node at (9,-0.4) {$\mathbb{Z}/2\mathbb{Z}$};
			\node at (11,-0.4) {$\mathbb{Z}/2\mathbb{Z}$};		
			\draw (7,0)--(9,0);	
			\draw (9,0)--(11,0);
			\draw (11,0)--(11,2);
			\draw (11,2)--(9,0);		
		\end{tikzpicture}
	\caption{Isomorphic graph products $G_{\Gamma}$ and $G_{\Gamma'}$.}
	\end{center}
\end{figure}

If the vertex groups in a graph product are cyclic, then we do not distinguish between a vertex $v$ and a generator of the vertex group $G_v$, so $G_v=\langle v\rangle$.

For graph products of cyclic groups, the no SIL condition on the graph can be used to show certain types of automorphism for the graph product do not occur. 
\begin{lemma}(\cite[Prop. 5.5]{CorredorGutierrez})
Let $G_\Gamma$ be a graph product of cyclic groups where the finite vertex groups have prime power orders.
Let $u, v, w$ be pairwise distinct vertices in $V(\Gamma)$, the vertex group $G_u$ be infinite cyclic and $\left\{v,w\right\}\notin E(\Gamma)$. If $lk(u)\subseteq st(v), st(w)$ then the map 
	\[h\mapsto \left\{
\begin{array}{ll}
	hvwv^{-1}w^{-1} &\text{if } h=u \\
	h &\text{if } h\in V(\Gamma), h\neq u.  \\
\end{array}
\right. \]
 induces an automorphism of $G_\Gamma$ and is called a \emph{commutator transvection}.
 \end{lemma}

\begin{lemma}\label{noCommutatorTransvestions}
	Let $G_\Gamma$ denote a graph product of cyclic groups where the finite groups have prime power orders. If $\Gamma$ has no SIL, then ${\rm Aut}(G_\Gamma)$ contains no commutator transvections.
\end{lemma}
\begin{proof}
	This follows from \cite[Prop. 2.7]{Sale}. For the sake of completeness we include the proof here. We assume for contradiction that ${\rm Aut}(G_\Gamma)$ has a commutator transvection. Hence there exist pairwise disjoint vertices $u, v, w$, with $G_u\cong\mathbb{Z}, \left\{v,w\right\}\notin E(\Gamma)$ and $lk(u)\subseteq st(v), st(w)$. We note that $\left\{u,v\right\}\notin E(\Gamma), \left\{u,w\right\}\notin E(\Gamma)$, since if  $v\in lk(u)\subseteq st(w)$, then $v\in st(w)$ and $\left\{v,w\right\}\in E(\Gamma)$,  contradiction. Therefore we have $lk(u)\subseteq lk(v)\cap lk(w)$ and $(v,w\mid u)$ is a SIL.
\end{proof}

\begin{NewTheoremC}\label{fgabelian}	
	Let $G_\Gamma$ be a graph product of finitely generated abelian groups.  The following statements are equivalent:
	\begin{enumerate}
		\item The group $G_\Gamma$ is semicomplete.
		\item The graph $\Gamma$ does not have a separating star.
	\end{enumerate}
\end{NewTheoremC}
\begin{proof}
	If $\Gamma$ has a separating star, then $G_\Gamma$ cannot be semicomplete, since a separating star always induces a partial conjugation that is not a global conjugation (see Lemma \ref{hasSeparatingStar}). 
	
	So now suppose that $\Gamma$ does not have a separating star. Without loss of generality we can also assume that all vertex groups are cyclic and the finite vertex groups have prime power orders by Lemma \ref{NoSepStarSIL} and its preceding paragraph. Due to \cite[Thm. 5.1]{SaleSusse} we know that IA$(G_\Gamma)$ is generated by partial conjugations and commutator transvections. Since $\Gamma$ has no separating star, every partial conjugation already has to be a global conjugation. According to Lemma \ref{noCommutatorTransvestions}, commutator transvections can only exist if $\Gamma$ has a SIL. Now Proposition \ref{noSIL} shows that if $\Gamma$ does not have a separating star, then $\Gamma$ also has no SIL. Thus IA$(G_\Gamma)$ consists only of conjugations, which is what we had to show.
\end{proof}

Before we prove Proposition C and E from the introduction we need to define one final type of automorphism.
\begin{definition}
	Let $\Gamma$ be a graph with labeled vertices. An \textit{automorphism of the labeled graph} is an automorphism of the graph $\Gamma$ which respects labels, i.e. for every vertex $v\in V(\Gamma)$, the vertices $v$ and $f(v)$ have the same label. 
	
	For a graph product, these automorphisms of the labeled graph induce an automorphism of the graph product, which we call a \textit{graph automorphism} by slight abuse of notation, since the automorphism group of the labeled graph naturally embeds into the automorphism group of the graph product.
\end{definition}
\begin{proof}[Proof of Proposition E]
	
	If $G_\Gamma$ is semicomplete, then Lemma \ref{abelian} implies that all vertex groups are abelian. So all that remains to show is the other direction.
	
	For the other direction we invoke \cite[Cor. 8.2]{Genevois}. This implies that Aut$(G_\Gamma)$ is generated by three types of automorphisms, graph automorphisms, partial conjugations and local automorphisms of the vertex groups. If $\Phi\colon {\rm Aut}(G_\Gamma)\to {\rm Aut}(G_\Gamma^{ab})$ is the induced homomorphism, then graph automorphisms and local automorphisms are not in the kernel of $\Phi$, see \cite[Lemma 4.1]{GenevoisMartin}. Thus the kernel consists only of partial conjugations. Since we assume that $\Gamma$ has no separating star, each partial conjugation is a global conjugation which means that $G_\Gamma$ is semicomplete. Note that Theorem \ref{fgabelian} does not apply here since the vertex groups are not necessarily finitely generated.
\end{proof}

\begin{proof}[Proof of Proposition C]
	Due to the definition of $\Delta$,  it is straightforward to verify that the center of $G_\Gamma$ is given by $G_\Delta$. 
	 
	Since $Z(G_\Gamma)=G_\Delta$ is characteristic, every $f\in {\rm Aut}(G_\Gamma)$ induces two automorphisms, 
	$$f'\colon G_\Delta\to G_\Delta\text{ and }f''\colon G_\Gamma/G_\Delta\to G_\Gamma/G_\Delta.$$ 
	
	Because of the definition of $\Delta$, we have 	
	$$G_\Gamma\cong G_\Delta\times G_{\Gamma-\Delta}\text{ and therefore }G_\Gamma/G_{\Delta}\cong G_{\Gamma-\Delta},$$
	 
	where $\Gamma-\Delta$ is the subgraph spanned by the vertex set $V(\Gamma)-V(\Delta)$.
	
	Let $\Phi\colon {\rm Aut}(G_\Gamma)\to {\rm Aut}(G^{ab}_\Gamma)$ denote the canonical map induced by the abelianization. 
	Suppose that $G_{\Gamma-\Delta}$ is semicomplete and let $f\in{\rm ker}(\Phi)$ denote an arbitrary element. We can write $f=(f',f'')$, where $f'\colon G_\Delta\to G_\Delta$ and $f''\colon G_{\Gamma-\Delta}\to G_{\Gamma-\Delta} $. Let $\Phi_{G_\Delta}$ and $\Phi_{G_{\Gamma-\Delta}}$ denote the restrictions of $\Phi$ to the factors of $G_\Gamma$, which is well defined since $G_\Delta$ is a characteristic subgroup. Since $f\in {\rm ker}(\Phi)$, we also have $f'\in {\rm ker}(\Phi_{G_\Delta})$ and $f''\in {\rm ker}(\Phi_{G_{\Gamma-\Delta}})$, $f''$ is a conjugation, since $G_{\Gamma-\Delta}$ is semicomplete. Since $G_\Delta$ is abelian, we can deduce that $f'=id$ (due to semicompleteness). Let $f''$ denote the conjugation by $a$. We show that $f$ is a conjugation by $a$. Given $x\in G_\Delta$ and $y \in G_{\Gamma-\Delta}$ arbitrary we obtain:
	 $$f(x,y)=(f'(x),f''(y))=(x,aya^{-1})=(axa^{-1},aya^{-1})$$
	 Thus $f$ is a conjugation and therefore and inner automorphism, showing $G_\Gamma$ is semicomplete.
	
	On the other hand if $G_\Gamma$ is semicomplete, we interpret $G_\Gamma$ as a different graph product, given by a graph with two vertices and one edge between them. The vertices are labeled $G_\Delta$ and $G_{\Gamma-\Delta}$, see Figure 9.

	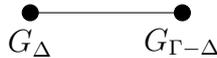
\begin{figure}[h!]
		\begin{center}
			\begin{tikzpicture}
				\draw[fill=black]  (0,0) circle (3pt);
				\draw[fill=black]  (2,0) circle (3pt);		
			\node at (0,-0.4) {$G_\Delta$}; 
				\node at (2,-0.4) {$G_{\Gamma-\Delta}$};		
			\draw (0,0)--(2,0);		
			\end{tikzpicture}
		\caption{Graph product $G_\Gamma$.}
		\end{center}
	\end{figure}

In this case we apply Lemma \ref{subgroup} to see, that $G_{\Gamma-\Delta}$ is semicomplete.
	
\end{proof}

%%%%%%%%%%%%%%%%%%%%%%%%%%%%%%%%%%%%%%%%%%%%%%%%%%%%%%%%%%%%%%%%%%%%%%%%%%%%%%%%%%%%%%%%%%%%%%%%%%%%%%%%%%%%%%%%%%%%%%%%%%%%%%%%%%%%%%%%%%%%%%%%%%%%%%%%%%

\section{Complete graph products of groups}

Many interesting groups have finite outer automorphism groups, i. e. the group ${\rm Out}(G):={\rm Aut}(G)/{\rm Inn}(G)$ is finite. For example, $\#{\rm Out}({\rm Sym}(n))=1$ for $n\neq 2,6$, $\# {\rm Out}({\rm Aut}(F_n))=1$ for $n\geq 3$ (see \cite{DyerFormanek}, \cite{BridsonVogtmann}) and  $\#{\rm Out}({\rm Aut}(\mathbb{Z}^n))\le 4$. Furthermore, the outer automorphism group of the mapping class group is also finite \cite{Ivanov}. 
\medskip

Let $G$ be a centerless group, then there is a natural embedding of $G$ into its automorphism group ${\rm Aut}(G)$, obtained by sending each $g\in G$ to the corresponding inner automorphism $f_g\in{\rm Aut}(G)$. In particular, ${\rm Aut}(G)$
is also a centerless group. Hence we define the automorphism tower of $G$ :
\[
G\hookrightarrow{\rm Aut}(G)\hookrightarrow{\rm Aut}({\rm Aut}(G))\hookrightarrow\ldots
\]
The automorphism tower terminates if there is a group in the tower which is isomorphic to its automorphism group by the above natural map. Such a group is called \emph{complete}. 

\begin{definition}
	A group $G$ is called \emph{complete}, if $Z(G)=\left\{1\right\}$ and ${\rm Aut}(G)={\rm Inn}(G)$.
\end{definition}

Note that if a group $G$ is complete, then $G$ is semicomplete. Further, if a group $G$ is perfect, i.e. $G=G'$, then $G$ is semicomplete if and only if ${\rm Aut}(G)={\rm Inn}(G)$. Hence, a centerless perfect group $G$ is complete if and only if $G$ is semicomplete.

 In this last paragraph we address the following question: 
\begin{center}
\textbf{When is a graph product of groups complete?}
\end{center}

Using the characterization of semicompleteness of graph products via combinatorial structure of graphs proven in Theorem D we can now show Proposition A from the introduction.

\begin{proof}[Proof of Proposition A]
We begin by showing (i) $\implies$ (ii), the other direction will follow easily from the observations made on the way.
	
First we note that the center of $G_\Gamma$ is trivial if and only if $\Gamma$ is not a star of a vertex or if $\Gamma$ is a star of a vertex $v$ then $Z(G_v)=1$. By assumption the vertex groups have non-trivial center, therefore $\Gamma$ is not a star of a vertex. 

If $G_\Gamma$ is complete, then $G_\Gamma$ is also semicomplete. Since $\Gamma$ is not a star of a vertex, Lemma \ref{abelian} tells us that all vertex groups are abelian and by Lemma \ref{hasSeparatingStar} it follows that $\Gamma$ has no separating star. For a vertex group $G_v$ we have always a factor automorphism induced by the map $g\mapsto -g$ for $g\in G_v$ and ${\rm id}_{G_w}$ for $w\neq v$. This factor automorphism is inner if and only if $g=-g$ for all $g\in G_v$. Thus, every non-trivial element in $G_v$ has order $2$ and since each vertex group is abelian and directly indecomposable, every vertex group is isomorphic to $\mathbb{Z}/2\mathbb{Z}$.

In was proven in \cite{Laurence} that the automorphism group of a graph product where the vertex groups are of order two is generated by: graph automorphisms, partial conjugations and dominated transvections, those are automorphisms defined as follows: Given two vertices $u, v, u\neq v$ such that $st(u)\subseteq st(v)$ then the dominated transvection is induced my the map $u\mapsto uv$, and $w\mapsto w$ for $w\neq u$. For more details and examples of these automorphisms we refer to \cite[Chap. 5]{Leder}. A dominated transvection is in ${\rm Aut}(G_{\Gamma})-{\rm Inn}(G_{\Gamma})$, since $G_\Gamma$ is a right-angled Coxeter group. This implies that if $uv=xux^{-1}$, we can apply the deletion condition to the right side, since it cannot be reduced. This reduces the length of the right hand side by $2$ every time, thus the length of the right hand side will always be an odd number. However the left hand side is the product of two generators and thus has even length.
Hence the graph $\Gamma$ does not have two vertices $u\neq v$  such that $st(u)\subseteq st(v)$. 

Summarizing our observations so far we have proved that if $G_\Gamma$ is complete, then $\Gamma$ is not a star of a vertex, all vertex groups are cyclic of order two and for two distinct vertices $v, w\in V(\Gamma)$ we have $st(v)$ is never contained in $st(w)$. 
	
Further, the graph $\Gamma$ is asymmetric, since a graph automorphism can never be an inner automorphism.

To see $\Gamma$ is connected we first note that $\Gamma$ has at most $2$ connected components, since it has no separating star. If it had $2$ connected components, then each component has to be a clique, due to no separating star condition. Furthermore those cliques can only consist of one vertex, because else the condition $st(v)\not\subseteq st(w)$ for $v\neq w$ is violated. Finally, we cannot have a graph consisting of precisely two vertices not connected by an edge since $\Gamma$ is asymmetric. 

Lastly, a connected graph that is asymmetric, has more than one vertex and has no separating star always has at least $7$ vertices, since an asymmetric graph with $6$ vertices always has a separating star, see Figure 2.

For the other direction we have that for any distinct pair of vertices $v, w\in V(\Gamma)$,  $st(v)$ is never a subset of $st(w)$, therefore $\Gamma$ is not a star of a vertex and we already know that the center of $G_\Gamma$ is trivial due to the first observation. As in the other direction we know that the only possible automorphisms are graph automorphisms, partial conjugations and dominated transvections. By the same argument as in the other direction we can only have partial conjugations. Partial conjugations are global conjugations when there are no separating stars, thus every automorphism is inner.
\end{proof}

\end{document}